\documentclass[12pt]{article}
\usepackage[utf8]{inputenc}
\usepackage{amssymb,amsmath} 
\usepackage{caption}
\usepackage{setspace}
\usepackage{graphicx,epsfig}
\usepackage{latexsym}
\usepackage{amssymb}
\usepackage{amsmath}
\usepackage{amsthm}
\usepackage{authblk}
\usepackage{enumerate}
\usepackage{geometry}
\usepackage{xcolor, soul}
\newtheorem{theorem}{Theorem}[section]
\newtheorem{definition}{Definition}[section]
\newtheorem{lemma}{Lemma}[section]

\newtheorem{proposition}{Proposition}[section]

\begin{document}
\title{Adjacency spectra of some subdivision hypergraphs}

\author[1]{\rm Anirban Banerjee}
\author[1]{\rm Arpita Das}

\affil[1]{\small Department of Mathematics and Statistics}
\affil[ ]{Indian Institute of Science Education and Research Kolkata}
\affil[ ]{Mohanpur-741246,  India}
\affil[ ]{\textit {{\scriptsize anirban.banerjee@iiserkol.ac.in, arpitamath@iiserkol.ac.in}}}
\maketitle

\begin{abstract}
Here, we define a subdivision operation for a hypergraph and compute all the eigenvalues of the subdivision of regular and certain non-regular hypergraphs. In non-regular hypergraphs, we investigate the power of regular graphs, various types of hyperflowers, and the squid-like hypergraph. Using our subdivision operation, we also show how to construct non-regular non-isomorphic cospectral hypergraphs.

\end{abstract}
\section{Introduction}
In spectral graph theory, the eigenvalues of a matrix associated with a graph are investigated, and from these eigenvalues, different properties of the graph structure are analyzed. There are several kinds of matrices associated with a graph, in which adjacency matrix,
Laplacian matrix and normalized Laplacian matrix are the popular matrices for researchers \cite{Bro,Car,Chu}. 
Similar to a graph, a hypergraph $H(V,E)$ also consists of a vertex set $V(H)$, and a hyperedge set $E(H)$. 
Unlike in a graph, a hyperedge of a hypergraph can be formed with more than two vertices. Thus in a hypergraph, $H(V,E)$, a hyperedge $e\in E(H)$ is a nonempty subset of the vertex set $V(H)$ \cite{Vol2}. Several properties of a hypergraph have been examined including Helly property, fractional transversal number, connectivity, and chromatic number in \cite{Ber,Vol1}. In 2012, Cooper and Dutle \cite{Cop} suggested the construction of adjacency tensors for studying hypergraphs. To avoid high  computational complexity in the tensor method, some authors have defined adjacency matrices for hypergraphs \cite{Ban, BanPar, Rod}.
Kumar and Varghese \cite{Kum} have proved some properties of the adjacency and Laplacian eigenvalues of a $(k, r)$-regular linear hypergraph and shown a relation between its dual with its line graph. 
Duttweiler and Reff.~\cite{Dut} have defined and studied oriented
hypergraph families that are analogous to signed graphs. They have also determined the adjacency and Laplacian eigenvalues for some of these families. 
In \cite{Sar}, Sarkar and Banerjee have found the complete spectrum of $s$-loose cycles and the characteristics polynomial of $s$-loose paths using edge corona operation. Moreover, using vertex corona, they have also shown how to generate infinitely many pairs of non-isomorphic co-spectral hypergraphs. Banerjee and Parui \cite{BanPar} introduced connectivity operators, namely, diffusion operators, general Laplacian operators, and general adjacency operators for hypergraphs. They have also derived spectral bounds for the weak connectivity number, degree of vertices, maximum cut, bipartition width, and isoperimetric constant of hypergraphs.\\
This article considers the matrix representation of hypergraphs suggested in \cite{Ban}. First, we define a subdivision operation for a hypergraph and determine the complete adjacency spectrum of the subdivision of regular and some non-regular hypergraphs. In non-regular hypergraphs, we study the power of regular graphs constructed in \cite{Hu},  $(l,r)$-hyperflower defined by Andreotti and Mulas in \cite{And}, and Hyperstar. 
We  also define a petal overlapped hyperflower and find the complete adjacency spectrum after applying the subdivision operation to it. Followed by squid \cite{Hu}, we construct a squid like hypergraph and determine the same of subdivision of this  hypergraph. Our derived results suggest a construction of non-regular non-isomorphic cospectral hypergraphs. This is discussed in the last section of our article.

Here, we study hypergraphs where the cardinality of each hyperedge is greater than or equal to two.
Two vertices $v_{i}, v_{j} \in V(H)$ are called adjacent if $v_{i}, v_{j} \in e$ for some $e\in E(H)$, and we denote it by $v_{i}\sim v_{j}$. The degree of a vertex $v\in V(H)$ is the number of hyperedges containing $v$. Here, in our work, we consider hypergraphs where the degree of any vertex is greater than zero.  A hypergraph $H$ is said to be a ($k$-)uniform hypergraph if all the hyperedges of $H$ have the same cardinality ($k$) and $H$ is called linear if each pair of hyperedges of $H$ has at most one vertex in common. A hypergraph $H$ is ($r$-)regular if the degree of all vertices is the same $(r)$.
Now, we recall the definition of the adjacency matrix $A_{H}$
of a hypergraph $H(V,E)$ from \cite{Ban}. For an $m$-uniform hypergraph the $(i,j)$-th entry $(A_{H})_{ij}$ of the adjacency matrix $A_{H}$ is defined as
$$
(A_{H})_{ij} = \left\{ \begin{array}{rl}
\frac{d_{ij}}{m-1} &\mbox{ if $v_i \sim v_j$} \\
  0 &\mbox{ elsewhere,}
       \end{array} \right.
$$


where $d_{ij}$ is the codegree of vertices $v_{i}$ and $v_{j}$, that is, the number of hyperedges containing both the vertices $v_{i}$ and $v_{j}$. \\
We denote the characteristic polynomial of $H$, (that is, of $A_{H}$) by $\Phi_{H}(x)=\det (xI-A_{H})$. So the roots of $\Phi_{H}(x)$ are the adjacency eigenvalues of hypergraph $H$. Since the matrix $A_{H}$ is symmetric, the roots of the polynomial $\Phi_{H}(x)$ are real. The eigenvalues of $A_{H}$ are denoted by $\lambda_{1}(H)\geq\lambda_{2}(H)\geq \cdots \geq \lambda_{n}(H)$. We can construct orthogonal eigenvectors of $A_{H}$.
Let $H$ be a uniform hypergraph whose incidence matrix is $B(H)$. If $H$ is an $r$-regular hypergraph then $B(H)B^{T}(H)=rI_{n}+A_{H}$. 
\\
The Kronecker product of the matrices $A=(a_{ij})$ of size $m \times n$ and $B$ of size $p \times q$ is denoted by $A \otimes B$, and is defined as an $mp \times nq$
partitioned matrix $(a_{ij}B)$. For the matrices $M$, $N$, $P$, and $Q$ of suitable sizes, $MN \otimes PQ=(M \otimes P)(N \otimes Q)$. Thus for nonsingular matrices $M$ and $N$, $(M \otimes N)^{-1}=M^{-1}\otimes N^{-1}$.  For square matrices $M$ and $N$ of order $k$ and $s$, respectively, $\det(M \otimes N)=(\det M)^s(\det N)^k$.\\
Throughout in this article, for any positive integers $k$, $m$, and $n$; $I_{k}$ denotes the identity matrix of size $k$, $J_{n\times m}$ denotes the $n\times m$ matrix whose all entries are $1$, $\mathbf{1}_{n}$ stands for the column vector of size $n$ with all the entries equal to $1$, $C_{n}$ denotes the cycle graph of order $n$, and $O_{m\times n}$ denotes the zero matrix of size $m\times n$.
\begin{lemma}{(Schur Complement)}: Suppose the order of all four matrices $M$, $N$, $P$ and $Q$ satisfy the rules of operations on matrices. Then we have,
\begin{eqnarray*}
\begin{vmatrix}
  M & N \\
  P & Q
  \end{vmatrix}
&=& |Q||M-NQ^{-1}P|, \mbox{if $Q$ is a non-singular square matrix},\\
&=& |M||Q-PM^{-1}N|, \mbox{if $M$ is a non-singular square matrix}.
\end{eqnarray*}
\end{lemma}
\noindent
\textbf{Equitable partition for hypergraphs:}\\
Let $H(V, E)$ be an $m$-uniform hypergraph. We say a partition $\Pi=\{C_{1}, C_{2}, \cdots, C_{k}\}$ of $V(H)$ is an equitable partition of $V(H)$ if for any $p, q\in \{1, 2, \cdots, k\}$ and for any $i \in C_{p}$,
$$\sum_{j,j\in C_{q}}(A_{H})_{ij}=b_{pq},$$
where $b_{pq}$ is a constant only depends on $p$ and $q$. 
The matrix $Q=(b_{pq})$ of size $k\times k$ is called a \textit{quotient matrix} of $A_{H}$ corresponding to the equitable partition $\Pi$.
Now we have the following proposition on the quotient matrix $Q$.
\begin{proposition}\cite{Bro}
Let $Q$ be a quotient matrix of any square matrix $A$ corresponding to an equitable partition. Then the spectrum of $A$ contains the spectrum of $Q$.
\end{proposition}

The concept of subdivision of graphs is uniquely defined in graph theory \cite{Cve}. Iradmusa \cite{Ira} defined the concept of subdivision of a hyperedge of a hypergraph. Here we introduce a different notion of subdivision of a hyperedge, as it is defined on an edge in subdivision of a graph. The subdivision graph of a graph $G$ is the graph obtained by replacing every edge in $G$ with a copy of $P_{3}$ (a path of length two). \\
\textbf{Subdivision of a hyperedge}: Subdivision of a hyperedge of cardinality $k$ is to introduce a new vertex and construct $k$ new hyperedges by taking the new vertex and $(k-1)$ old vertices from that hyperedge and delete the old hyperedge. As $(k-1)$ vertices can be choosen in ${k}\choose{k-1}$ ways from $k$ old vertices, hence by replacing one old hyperedge of cardinality $k$, we get $k$ number of new hyperedges with cardinality $k$.\\
\begin{figure}
\centering
\includegraphics[width=4.0in]{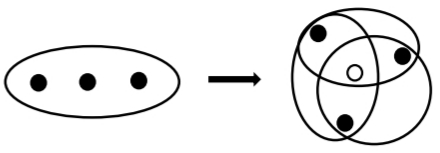}
\caption{Subdivision of an edge}
\end{figure}
\textbf{Subdivision hypergraph}: A subdivision hypergraph $S(H)$ of a hypergraph $H$ is a hypergraph resulting from the subdivision of all hyperedges in $H$. \\
Note that a subdivision hypergraph is not a linear hypergraph.
\section{Adjacency of subdivision of hypergraphs}
Now we find the spectrum of the adjacency matrix of some subdivision hypergraphs.
\subsection{Adjacency spectra of subdivision of regular hypergraph}
\begin{theorem}\label{t1}
Let $H$ be an $r$-regular $k$-uniform hypergraph with $n$ vertices and $m$ hyperedges. Then the adjacency spectrum of $S(H)$ consists of:
\begin{enumerate}[(i)]
\item the eigenvalue $0$ with the multiplicity $m-n$,
\item the roots of the equations \\
$(k-1)x^{2}-(k-2)\lambda_{i}(H)x-r(k-1)-(k-1)\lambda_{i}(H)=0$, for $i=1,2,\cdots,n$;\\
where $\lambda_{i}(H)$ are the eigenvalues of $A_H$.
\end{enumerate}
\end{theorem}
\begin{proof}
Let us consider a $k$-uniform linear hypergraph $H$ with $n$ vertices, $m$ hyperedges and regularity $r$.\\
The adjacency matrix of $S(H)$ is written as:
\begin{eqnarray*}
A_{S(H)}&=&\frac{1}{(k-1)}\begin{pmatrix}
  (k-2)A(H) & (k-1)B(H) \\
  (k-1)B^{T}(H) & O_{m}
  \end{pmatrix}.
\end{eqnarray*}
The adjacency characteristic polynomial of $S(H)$
\begin{eqnarray*}
\Phi_{A_{S(H)}}(x)&=&\det\begin{pmatrix}
  xI_{n}-\frac{k-2}{k-1}A(H) & -B(H) \\
  -B^{T}(H) & xI_{m}
  \end{pmatrix}\\
&=& x^{m}\det\Big(xI_{n}-\frac{k-2}{k-1}A(H)-\frac{1}{x}B(H)B^{T}(H)\Big)\\
&=& x^{m}\det\Big(xI_{n}-\frac{k-2}{k-1}A(H)-\frac{1}{x}(A(H)+rI_{n}\Big)\\
&=& x^{m}\det\Big\{\Big(x-\frac{r}{x}\Big)I_{n}-\Big(\frac{k-2}{k-1}+\frac{1}{x}\Big)A(H)\Big\}\\
&=& x^{m}\prod_{i=1}^{n}\Big\{x-\frac{r}{x}-\Big(\frac{k-2}{k-1}+\frac{1}{x}\Big)\lambda_{i}(H)\Big\}\\
&=& x^{m-n}\prod_{i=1}^{n}\Big\{x^{2}-\frac{k-2}{k-1}\lambda_{i}(H)x-r-\lambda_{i}(H)\Big\}
\end{eqnarray*}
Thus the proof follows.
\end{proof}
Now, we consider hypergraphs which are not regular.
\subsection{Adjacency spectra of subdivision of power of a graph}
We recall the definition of \textit{power of graph} from \cite{Hu}.
\begin{definition}
Let $G=(V,E)$ be a graph. For any $k\geq3$, the $k^{th}$ power of $G$, $G^{k}=(V^{k},E^{k})$ is defined as the k-uniform hypergraph with the set of vertices $V^{k}(H)=V(H)\cup\{i_{e,1},\cdots,i_{e,k-2}|e\in E(H)\}$ and the set of edges $E^{k}(H)=\{e\cup\{i_{e,1},\cdots,i_{e,k-2}\}|e\in E(H)\}$.
\end{definition}
\begin{figure}
\centering
\includegraphics[width=4.0in]{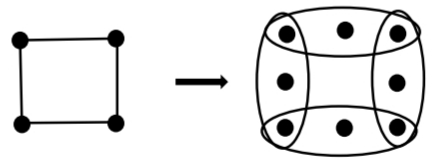}
\caption{A power of a cycle graph: Linear Hypercycle}
\end{figure}
\begin{theorem}\label{t2}
Let $G$ be an $r$-regular graph with $n$ vertices and $m$ edges. Then the adjacency spectrum of $S(G^{k})$ consists of:
\begin{enumerate}[(i)]
\item the eigenvalue $-\frac{k-2}{k-1}$ with the multiplicity $m(k-3)$,
\item the two roots of the equation \\
$(k-1)x^{2}-(k-2)(k-3)x-(k-1)(k-2)=0$, with multiplicity $m-n$, and
\item the roots of the equations\\
$(k-1)^{2}x^{3}-(k-1)(k-2)(\lambda_{i}(G)+k-3)x^{2}-\{(k-1)^{2}(r+\lambda_{i}(G))+(k-1)^{2}(k-2)+(k-2)^{3}(r+\lambda_{i}(G))-(k-2)^{2}(k-3)\lambda_{i}(G)\}x+(k-1)(k-2)(k-3)(r+\lambda_{i}(G))+(k-1)(k-2)^{2}\lambda_{i}(G)-2(k-1)(k-2)^{2}(r+\lambda_{i}(G))=0$, for $i=1,2,\cdots,n$; \\
where $\lambda_{i}(G)$ are the eigenvalues of $A_G$.
\end{enumerate}
\end{theorem}
\begin{proof}
Let us consider a $r$-regular graph $G$ with $n$ vertices and $m$ edges.\\
Then the adjacency matrix of $S(G^{k})$ can be written as:
\begin{eqnarray*}
A_{S(G^{k})}&=&\frac{1}{(k-1)}\begin{pmatrix}
  (k-2)A(G) & (k-2)\mathbf{1}_{k-2}^{T}\otimes B(G) & (k-1)B(G) \\
  (k-2)\mathbf{1}_{k-2}\otimes B(G)^{T} & (k-2)(J_{k-2}-I_{k-2})\otimes I_{m} & (k-1)\mathbf{1}_{k-2}\otimes I_{m}\\
  (k-1)B^{T}(G) & (k-1)\mathbf{1}_{k-2}^{T}\otimes I_{m} & O_{m}
  \end{pmatrix}.
\end{eqnarray*}
The adjacency characteristic polynomial of $S(G^{k})$
\begin{eqnarray*}
\Phi_{A_{S(G^{k})}}(x)&=&\det\begin{pmatrix}
  xI_{n}-\frac{k-2}{k-1}A(G) & -\frac{k-2}{k-1}\mathbf{1}_{k-2}^{T}\otimes B(G) & -B(G) \\
  -\frac{k-2}{k-1}\mathbf{1}_{k-2}\otimes B(G)^{T} & \{xI_{k-2}-\frac{k-2}{k-1}(J_{k-2}-I_{k-2})\}\otimes I_{m} & -\mathbf{1}_{k-2}\otimes I_{m}\\
  -B^{T}(G) & -\mathbf{1}_{k-2}^{T}\otimes I_{m} & xI_{m}
  \end{pmatrix}\\\\
&=& x^{m}\det M, 
\end{eqnarray*}
where
\begin{eqnarray*}
M&=&\Bigg[\begin{pmatrix}
xI_{n}-\frac{k-2}{k-1}A(G) & -\frac{k-2}{k-1}\mathbf{1}_{k-2}^{T}\otimes B(G) \\
-\frac{k-2}{k-1}\mathbf{1}_{k-2}\otimes B(G)^{T} & \{xI_{k-2}+\frac{k-2}{k-1}(I_{k-2}-J_{k-2})\}\otimes I_{m}
\end{pmatrix}
\\\\&&-
\begin{pmatrix}
-B(G)\\
-\mathbf{1}_{k-2}\otimes I_{m}
\end{pmatrix}
(xI_{m})^{-1}
\begin{pmatrix}
-B^{T}(G) & -\mathbf{1}_{k-2}^{T}\otimes I_{m}
\end{pmatrix}\Bigg]\\\\
&=& \begin{pmatrix}
xI_{n}-\frac{k-2}{k-1}A(G)-\frac{1}{x}B(G)B(G)^{T} & -(\frac{k-2}{k-1}+\frac{1}{x})\mathbf{1}_{k-2}^{T}\otimes B(G) \\
-(\frac{k-2}{k-1}+\frac{1}{x})\mathbf{1}_{k-2}\otimes B(G)^{T} & \{(x+\frac{k-2}{k-1})I_{k-2}-(\frac{k-2}{k-1}+\frac{1}{x})J_{k-2}\}\otimes I_{m}
\end{pmatrix}\\\\
&=& \begin{pmatrix}
xI_{n}-\frac{k-2}{k-1}A(G)-\frac{1}{x}(rI_{n}+A(G)) & -(\frac{k-2}{k-1}+\frac{1}{x})\mathbf{1}_{k-2}^{T}\otimes B(G) \\
-(\frac{k-2}{k-1}+\frac{1}{x})\mathbf{1}_{k-2}\otimes B(G)^{T} & \{(x+\frac{k-2}{k-1})I_{k-2}-(\frac{k-2}{k-1}+\frac{1}{x})J_{k-2}\}\otimes I_{m}
\end{pmatrix}
\end{eqnarray*}
Therefore
\begin{eqnarray*}
\Phi_{A_{S(G^{k})}}(x)&=& x^{m}\det\Big[\Big\{\Big(x+\frac{k-2}{k-1}\Big)I_{k-2}-\Big(\frac{k-2}{k-1}+\frac{1}{x}\Big)J_{k-2}\Big\}\otimes I_{m}\Big]\\
&&\det\Bigg[xI_{n}-\frac{k-2}{k-1}A(G)-\frac{1}{x}(rI_{n}+A(G))-\Big(\frac{k-2}{k-1}+\frac{1}{x}\Big)^{2}\\
&&\Bigg\{(\mathbf{1}_{k-2}^{T}\otimes B(G))\Big\{\Big(\Big(x+\frac{k-2}{k-1}\Big)I_{k-2}-\Big(\frac{k-2}{k-1}+\frac{1}{x}\Big)J_{k-2}\Big)\otimes I_{m}\Big\}^{-1}(\mathbf{1}_{k-2}\otimes B(G)^{T})\Bigg\}\Bigg]\\\\
&=& x^{m}\det\Big[\Big(x+\frac{k-2}{k-1}\Big)I_{k-2}-\Big(\frac{k-2}{k-1}+\frac{1}{x}\Big)J_{k-2}\Big]^{m}\\
&&\det\Bigg[\Big(x-\frac{r}{x}\Big)I_{n}-\Big(\frac{k-2}{k-1}+\frac{1}{x}\Big)A(G)-\Big(\frac{k-2}{k-1}+\frac{1}{x}\Big)^{2}\\
&&\Big\{\mathbf{1}_{k-2}^{T} \Big(\Big(x+\frac{k-2}{k-1}\Big)I_{k-2}-\Big(\frac{k-2}{k-1}+\frac{1}{x}\Big)J_{k-2}\Big)^{-1}\mathbf{1}_{k-2}\Big\}\otimes B(G)B(G)^{T}\Bigg]\\\\
&=& \Big(x+\frac{k-2}{k-1}\Big)^{m(k-3)}\Big\{\Big(x+\frac{k-2}{k-1}\Big)-\Big(\frac{k-2}{k-1}+\frac{1}{x}\Big)(k-2)\Big\}^{m}\\
&&\det\Bigg[\Big(x-\frac{r}{x}\Big)I_{n}-\Big(\frac{k-2}{k-1}+\frac{1}{x}\Big)A(G)-\Big(\frac{k-2}{k-1}+\frac{1}{x}\Big)^{2} \frac{(k-2)(rI_{n}+A(G)}{\Big(x+\frac{k-2}{k-1}\Big)-\Big(\frac{k-2}{k-1}+\frac{1}{x}\Big)(k-2)}\Bigg]\\\\
&=& \Big(x+\frac{k-2}{k-1}\Big)^{m(k-3)}\Big[x^{2}-\frac{(k-2)(k-3)}{k-1}x-(k-2)\Big]^{m}\\
&&\prod_{i=1}^{n}x^{-1}\Bigg[x^2-r-\Big(\frac{k-2}{k-1}x+1\Big)\lambda_{i}(G)-\Big(\frac{k-2}{k-1}x+1\Big)^{2} \frac{(k-2)(r+\lambda_{i}(G))}{x^{2}-\frac{(k-2)(k-3)}{k-1}x-(k-2)}\Bigg]\\\\
&=& \Big(x+\frac{k-2}{k-1}\Big)^{m(k-3)}\Big\{x^{2}-\frac{(k-2)(k-3)}{k-1}x-(k-2)\Big\}^{m-n}\prod_{i=1}^{n}x^{-1}\Bigg[\Big\{x^2-\frac{k-2}{k-1}\lambda_{i}(G)x\\&&-r-\lambda_{i}(G)\Big\}
\Big\{x^{2}-\frac{(k-2)(k-3)}{k-1}x-(k-2)\Big\}-\Big(\frac{k-2}{k-1}x+1\Big)^{2} (k-2)(r+\lambda_{i}(G))\Bigg]\\\\
&=& \Big(x+\frac{k-2}{k-1}\Big)^{m(k-3)}\Big\{x^{2}-\frac{(k-2)(k-3)}{k-1}x-(k-2)\Big\}^{m-n}\prod_{i=1}^{n}\Big[x^{3}-\frac{k-2}{k-1}(\lambda_{i}(G)+k-3)x^{2}-\\&&\Big\{r+\lambda_{i}(G)+k-2+\frac{(k-2)^{3}}{(k-1)^2}(r+\lambda_{i}(G))-\frac{(k-2)^{2}(k-3)}{(k-1)^2}\lambda_{i}(G)\Big\}x\\&&+\frac{(k-2)(k-3)(r+\lambda_{i}(G))}{k-1}+\frac{(k-2)^{2}}{k-1}\lambda_{i}(G)-\frac{2(k-2)^{2}}{k-1}(r+\lambda_{i}(G))\Big]
\end{eqnarray*}
Thus the proof follows.
\end{proof}
\subsection{Adjacency spectra of subdivision of hyperflower}
First we recall the definition of hyperflower from \cite{And}.
\begin{definition}
A $(l, r)$-hyperflower with $t$ twins is a hypergraph $\Gamma=(V, E)$ whose vertex set can be written as $V=U \cup W$, where $U$ is a set of $l.t$ vertices $u_{1}^{1},\cdots, u_{t}^{1}, \cdots, u_{l}^{1}, \cdots , u_{l}^{t}$ which are called peripheral and there exist $r$ disjoint sets of vertices $h_{1},\cdots, h_{r}$ such that $W=\cup_{k=1}^{r}h_{k}$ and $E=\big\{e|e=h_{k}\cup \{u_{1}^{j},\cdots, u_{t}^{j}\}$ for $j=1,\cdots,l$ and $k=1,\cdots,r\big\}$.
\end{definition}
If $r=1$, we simply say that $\Gamma$ is a $l$-hyperflower with $t$ twins.
We denote a $l$-hyperflower with $t$ twins and $|h_{1}|=s$ by $F_{l}^{\{s,t\}}$. 
\begin{figure}
\centering
\includegraphics[width=2.0in]{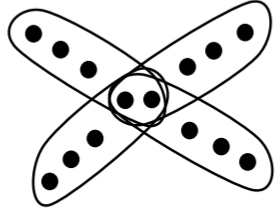}
\caption{4-hyperflower with 3 twins and 2 center}
\end{figure}
\begin{theorem}\label{t3}
The spectrum of $S(F_{l}^{\{s,t\}})$ consists of:
\begin{enumerate}[(i)]
\item the eigenvalue $-\frac{k-2}{k-1}$ with multiplicity $l(t-1)$,
\item the eigenvalue $-\frac{l(k-2)}{k-1}$ with multiplicity $(s-1)$,
\item the two eigenvalues $\frac{(t-1)(k-2)\pm\sqrt{(t-1)^{2}(k-2)^{2}+4t(k-1)^{2}}}{2(k-1)}$, each with multiplicity $(l-1)$,
\item the three roots of the equation \\
$(k-1)^{2}\lambda^{3}-(k-1)(k-2)(ls+t-l-1)\lambda^{2}-\{(ls+t)(k-1)^{2}+l(s+t-1)(k-2)^{2}\}\lambda-l(s+t)(k-1)(k-2)=0$.
\end{enumerate}
\end{theorem}
\begin{proof}
Let $E(F_{l}^{\{s,t\}})=\{e_{1},\cdots,e_{l}\}$, let $U_{j}=\{u_{1}^{j},\cdots,u_{t}^{j}\}$ be the set of vertices of degree one in $e_{j}$, let $W=\{w_{1},\cdots,w_{s}\}$ be the vertices of degree $l$ in $F_{l}^{\{s,t\}}$ and let $P=\{p_{1},\cdots,p_{l}\}$, where $p_{j}$ be the new vertex introduced while subdividing the hyperedge $e_{j}\in E(F_{l}^{\{s,t\}})$.\\
Let $k=s+t$. If we partition the vertices of $S(F_{l}^{\{s,t\}})$ as $\{w_{1},\cdots,w_{s}\}\cup\{u_{1}^{1},\cdots,u_{t}^{1}\}\cup\cdots\cup\{u_{1}^{l},\cdots,u_{t}^{l}\}\cup\{p_{1},\cdots,p_{l}\},$
then the adjacency matrix of $(F_{l}^{\{s,t\}})$ can be written as:
\begin{equation*}
A_{S(F_{l}^{\{s,t\}})}=\frac{1}{(k-1)}\begin{pmatrix}
  l(k-2)(J_{s}-I_{s}) & (k-2)J_{s\times t}\otimes\mathbf{1}_{l}^{T} & (k-1)J_{s\times l} \\
  (k-2)J_{t \times s}\otimes\mathbf{1}_{l} & (k-2)(J_{t}-I_{t})\otimes I_{l} & (k-1)\mathbf{1}_{t}\otimes I_{l}\\
  (k-1)J_{l\times s} & (k-1)\mathbf{1}_{t}^{T}\otimes I_{l} & O_{l}
  \end{pmatrix}.
\end{equation*}
We can construct the following family of $l(t-1)$ linearly independent vectors.
\begin{equation}\label{e1}
 x_{i}^{\{e_{j}\}}(v)=\left\{ \begin{array}{rl}
& 1  \mbox{~~~for $v=u_{1}^{j}$}\\
& -1 \mbox{~~~ for $v=u_{i}^{j}$} \\
& 0 \mbox{~~~elsewhere},
       \end{array} \right.
\end{equation} for $2\leq i\leq t$.
Clearly $$A_{S(F_{l}^{\{s,t\}})}x_{i}^{\{e_{j}\}}=-\frac{k-2}{k-1}x_{i}^{\{e_{j}\}}.$$
Repeating this construction for the other hyperedges of $S(F_{l}^{\{s,t\}})$, we obtain $l(t-1)$ linearly independent eigenvectors, associated with the eigenvalue $\lambda=-\frac{k-2}{k-1}$.\\
Now we can construct the following family of $s-1$ linearly independent vectors.
\begin{equation}\label{e2}
 y_{i}(v)=\left\{ \begin{array}{rl}
& 1  \mbox{~~~for $v=w_{1}$}\\
& -1 \mbox{~~~for $v=w_{i}$} \\
& 0 \mbox{~~~elsewhere},
       \end{array} \right.
\end{equation} for $2\leq i\leq s$.
Clearly, any vector in (\ref{e2}) is orthogonal to any other vector in (\ref{e1}) and $$A_{S(F_{l}^{\{s,t\}})}y_{i}=-\frac{l(k-2)}{k-1}y_{i}.$$
Thus we obtain $(s-1)$ linearly independent eigenvectors, associated with the eigenvalue $\lambda=-\frac{l(k-2)}{k-1}$.\\
Let us construct $2(l-1)$ linearly independent vectors which are orthogonal to the previously obtained eigenvectors. The construction of these vectors are as follows. For $j=2,\cdots,l$ we have
 \begin{equation}\label{e3}
f_{j}(v)=\left\{ \begin{array}{rl}
& c_{1} \mbox{~~~ for $v=u_{i}^{1};i=1,\cdots,t$}\\
& -c_{1} \mbox{~~~ for $v=u_{i}^{j};i=1,\cdots,t$}\\
& c_{2} \mbox{~~~for $v=p_{1}$}\\
& -c_{2} \mbox{~~~for $v=p_{j}$}\\
& 0 \mbox{~~~elsewhere},
       \end{array} \right.
\end{equation}
where $c_{1}, c_{2}\in R$. If $f_{j}$ is an eigenvector of $A_{S(F_{l}^{\{s,t\}})}$ with an eigenvalue $\lambda$ then we have
 $$A_{S(F_{l}^{\{s,t\}})}f_{j}=\lambda f_{j},$$ which provides us the following equations,
$$
\left. \begin{array}{rl}
&\lambda c_{1}=\frac{(t-1)(k-2)}{k-1}c_{1}+c_{2} \\
&\lambda c_{2}=tc_{1}
       \end{array} \right\}
$$
for $j=2,\cdots,l$. Eliminating $c_{1}$ and $c_{2}$ we have the following equation for the eigenvalue $\lambda$,\\
$\lambda^{2}-\frac{(t-1)(k-2)}{k-1}\lambda-t=0$.\\
$\Rightarrow$ $\lambda=\frac{(t-1)(k-2)\pm\sqrt{(t-1)^{2}(k-2)^{2}+4t(k-1)^{2}}}{2(k-1)}$.\\ 
Thus we find $\lambda$ is an eigenvalue of $S(F_{l}^{\{s,t\}})$, with the multiplicity $(l-1)$.\\
Let $\{W,U,P\}$ be a partition of $V(S(F_{l}^{\{s,t\}}))$. Clearly this is an equitable partition of $V(S(F_{l}^{\{s,t\}}))$ where the quotient matrix $Q$ is defined as:
\begin{equation*}
Q=\frac{1}{(k-1)}\begin{pmatrix}
l(s-1)(k-2) & lt(k-2) & l(k-1)\\
s(k-2) & (t-1)(k-2) & (k-1) \\
s(k-1) & t(k-1) & O
  \end{pmatrix}.
\end{equation*}
Let $(\lambda_{k'},g'_{k'})$ be an eigenpair of $Q$ for $k'=1,2,3$. Using the matrix $Q$ we get $\lambda_{k'}$ $(k'=1,2,3)$ are the roots of the following equation.
\begin{eqnarray}\label{e4}
&&(k-1)^{2}\lambda^{3}-(k-1)(k-2)(ls+t-l-1)\lambda^{2}-\{(ls+t)(k-1)^{2}+l(s+t-1)(k-2)^{2}\}\lambda\notag\\&&-l(s+t)(k-1)(k-2)=0.
\end{eqnarray}
Let us define three vectors $g_{k'}$ for $k'=1,2,3$ as follows:
\begin{equation}\label{e5}
g_{k'}(v)=\left\{ \begin{array}{rl}
& g'_{k'}(1) \mbox{~~~for $v=w_{i};i=1,\cdots,s$}\\
& g'_{k'}(2) \mbox{~~~ for $v=u_{i}^{j};i=1,\cdots,t,j=1,\cdots,l$} \\
& g'_{k'}(3) \mbox{~~~ for $v=p_{j};j=1,\cdots,l$}.
       \end{array} \right.
\end{equation}
Clearly any vector in (\ref{e5}) is orthogonal to any other vector in (\ref{e1}), (\ref{e2}) and (\ref{e3}) and satisfy the equation $$A_{S(F_{l}^{\{s,t\}})}g_{k'}=\lambda_{k'} g_{k'}, ~\mbox{for}~ k'=1,2,3.$$
Thus $\lambda_{k'}$ $(k'=1,2,3)$ are the remaining three eigenvalues of $A_{S(F_{l}^{\{s,t\}})}$ which can be derived from the equation (\ref{e4}).
\end{proof}
\subsubsection{Adjacency spectra of subdivision of hyperstar}
Hyperstar is a particular type of hyperflower. If $s=1$, then $k$ uniform $l$-hyperflower with $t$ twins is a hyperstar where $t=k-1$.
\begin{theorem}\label{t4}
The spectrum of the above hyperstar consists of:
\begin{enumerate}[(i)]
\item $-\frac{k-2}{k-1}$ with the multiplicity $l(k-2)$,
\item $\frac{(k-2)^{2}\pm\sqrt{(k-2)^{4}+4(k-1)^{3}}}{2(k-1)}$, each with multiplicity $(l-1)$, and
\item the three roots of the equation \\
$(k-1)\lambda^{3}-(k-2)^{2}\lambda^{2}-\{l(k-1)+l(k-2)^{2}+(k-1)^{2}\}\lambda-lk(k-2)=0$.
\end{enumerate}
\end{theorem}
\subsubsection{Adjacency spectra of subdivision of petal overlapped hyperflower}
\begin{definition}
A $l$-hyperflower with $t$ twins is said petal overlapped if each petal hyperedge has a common vertex with the previous and following petal hyperedges respectively. Thus if $u_{i}^{j}$ and $w_{j}$ denote the vertex of degree one and $l$, respectively, and $v_{j}$ is the common vertex of any two petal hyperedge then each hyperedge $e_{j}$ consists of the vertices $\{w_{1},\cdots,w_{s},u_{1}^{j},\cdots,u_{t}^{j},v_{j},v_{j+1}\}$.
\end{definition}
\begin{figure}
\centering
\includegraphics[width=2.0in]{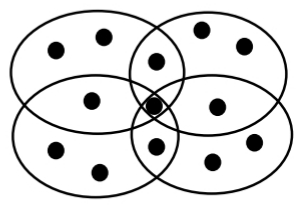}
\caption{petal overlapped 4-Hyperflower with 2 twins}
\end{figure}
\begin{theorem}\label{t5}
Let $F_{l}^{\{s,t\}}$ be a petal overlapped $l$-hyperflower with $t$ twins. Then the spectrum of $S(F_{l}^{\{s,t\}})$ consists of:
\begin{enumerate}[(i)]
\item the eigenvalue $-\frac{k-2}{k-1}$ with the multiplicity $l(t-1)$,
\item the eigenvalue $-\frac{l(k-2)}{k-1}$ with the multiplicity $(s-1)$,
\item the four roots of the equation \\
$(k-1)^{3}\lambda^{4}-(k-1)^{2}(k-2)(ls+t+1-l)\lambda^{3}-(k-1)\{(ls+t+4)(k-1)^{2}+(3ls+lt+l+2t+2)(k-2)^{2}\}\lambda^{2}-\{2l(s+t+1)(k-2)^{3}+(3ls+lt+4l+2t+4)(k-1)^{2}(k-2)\}\lambda-2l(s+t+2)(k-1)(k-2)^{2}=0$, and
\item the roots of the following equations \\
$(k-1)^{2}\lambda^{3}-(k-1)(k-2)\{(\alpha^{-1}+\alpha)+(t-1)\}\lambda^{2}-\{t(k-1)^{2}+(\alpha^{-1}+2+\alpha)(k-1)^{2}+(\alpha^{-1}+2+\alpha)t(k-2)^{2}-(\alpha^{-1}+\alpha)(t-1)(k-2)^{2}\}\lambda-(\alpha^{-1}+2+\alpha)(k-1)(k-2)-2t(k-1)(k-2)=0$\\
where $\alpha^{-1}+\alpha=2cos\frac{2\pi j'}{l}$, $j'=1,\cdots,l-1$.
\end{enumerate}
\end{theorem}
\begin{proof}
Let $E(F_{l}^{\{s,t\}})=\{e_{1},\cdots,e_{l}\}$, let $U_{j}=\{u_{1}^{j},\cdots,u_{t}^{j}\}$ be the vertices of degree one in $e_{j}$, $V=\{v_{1},\cdots,v_{l}\}$ be the vertices of degree two, $W=\{w_{1},\cdots,w_{s}\}$ be the vertices of degree $l$ in $F_{n}$. Let $P=\{p_{1},\cdots,p_{l}\}$, where $p_{j}$ be the new vertex introduced while subdividing the hyperedge $e_{j}\in E(F_{l}^{\{s,t\}})$ .\\
Take $k=s+t+2$.\\
If we partition the vertices of $S(F_{l}^{\{s,t\}})$ as $\{w_{1},\cdots,w_{s}\}\cup\{v_{1},\cdots,v_{l}\}\cup\{u_{1}^{1},\cdots,u_{t}^{1}\}\cup\cdots\cup\{u_{1}^{l},\cdots,u_{t}^{l}\}\cup\{p_{1},\cdots,p_{l}\},$
then the adjacency matrix of $S(F_{l}^{\{s,t\}})$ can be written as:
\begin{equation*}
A_{S(F_{l}^{\{s,t\}})}=\frac{1}{(k-1)}\begin{pmatrix}
  l(k-2)(J_{s}-I_{s}) & 2(k-2)J_{s\times l} & (k-2)J_{s\times lt} & (k-1)J_{s\times l} \\
  2(k-2)J_{l\times s} & (k-2)A(C_{l}) & (k-2)\mathbf{1}_{t}^{T}\otimes B(C_{l}) & (k-1)B(C_{l}) \\
  (k-2)J_{lt\times s} & (k-2)\mathbf{1}_{t}\otimes B^{T}(C_{l}) & (J_{t}-I_{t})\otimes I_{l} & (k-1)\mathbf{1}_{t}\otimes I_{l}\\
  (k-1)J_{l\times s} & (k-1)B^{T}(C_{l}) & (k-1)\mathbf{1}_{t}^{T}\otimes I_{l} & O_{l}
  \end{pmatrix}.
\end{equation*}
We can construct the following family of $l(t-1)$ linearly independent vectors.
\begin{equation}\label{e6}
 x_{i}^{\{e_{j}\}}(v)=\left\{ \begin{array}{rl}
& 1  \mbox{~~~for $v=u_{1}^{j}$}\\
& -1 \mbox{~~~ for $v=u_{i}^{j}$} \\
& 0 \mbox{~~~elsewhere},
       \end{array} \right.
\end{equation} for $2\leq i\leq t$ and $j=1,\cdots,l$.
Using the matrix $A_{S(F_{l}^{\{s,t\}})}$ and equation (\ref{e6}) we have $$A_{S(F_{l}^{\{s,t\}})}x_{i}^{\{e_{j}\}}=-\frac{k-2}{k-1}x_{i}^{\{e_{j}\}}, ~\mbox{for}~ i=2,\cdots,t ~\mbox{and}~ j=1,\cdots,l.$$
Repeating this construction for the other hyperedges of $S(F_{l}^{\{s,t\}})$, we obtain $l(t-1)$ linearly independent eigenvectors, associated with the eigenvalue $\lambda=-\frac{k-2}{k-1}$.\\
We construct the following family of $s-1$ linearly independent vectors.
\begin{equation}\label{e15}
 y_{i}(v)=\left\{ \begin{array}{rl}
& 1  \mbox{~~~for $v=w_{1}$}\\
& -1 \mbox{~~~for $v=w_{i}$} \\
& 0 \mbox{~~~elsewhere},
       \end{array} \right.
\end{equation} for $2\leq i\leq s$.
Clearly, any vector in (\ref{e15}) is orthogonal to any other vector in (\ref{e6}) and $$A_{S(F_{l}^{\{s,t\}})}y_{i}=-\frac{l(k-2)}{k-1}y_{i}.$$
Thus we obtain $(s-1)$ linearly independent eigenvectors, associated with the eigenvalue $\lambda=-\frac{l(k-2)}{k-1}$.\\
Let $\{W,V,U,P\}$ be a partition of $V(S(F_{l}^{\{s,t\}}))$. Clearly this is an equitable partition of $V(S(F_{l}^{\{s,t\}}))$ where the quotient matrix $Q$ is defined as:
\begin{eqnarray*}
Q&=&\frac{1}{(k-1)}\begin{pmatrix}
l(s-1)(k-2) & 2l(k-2) & lt(k-2) & l(k-1)\\
2s(k-2) & 2(k-2) & 2t(k-2) & 2(k-1) \\
s(k-2) & 2(k-2) & (t-1)(k-2) & (k-1)\\
s(k-1) & 2(k-1) & t(k-1) & O
  \end{pmatrix}.
\end{eqnarray*}
Let $(\lambda_{k'},g'_{k'})$ be an eigenpair of $Q$ for $k'=1,2,3,4$. Using the matrix $Q$ we get $\lambda_{k'}$ $(k'=1,2,3)$ are the roots of the following equation.
\begin{eqnarray}\label{e7}
&&(k-1)^{3}\lambda^{4}-(k-1)^{2}(k-2)(ls+t+1-l)\lambda^{3}-(k-1)\{(ls+t+4)(k-1)^{2}\notag\\&&+(3ls+lt+l+2t+2)(k-2)^{2}\}\lambda^{2}-\{2l(s+t+1)(k-2)^{3}\notag\\&&+(3ls+lt+4l+2t+4)(k-1)^{2}(k-2)\}\lambda-2l(s+t+2)(k-1)(k-2)^{2}=0.
\end{eqnarray}
Let us define four vectors $g_{k'}$ as
\begin{equation}\label{e8}
g_{k'}(v)=\left\{ \begin{array}{rl}
& g'_{k'}(1) \mbox{~~~for $v=w_{i};i=1,\cdots,s$}\\
& g'_{k'}(2) \mbox{~~~ for $v=v_{j};j=1,\cdots,l$} \\
& g'_{k'}(3) \mbox{~~~ for $v=u_{i}^{j};i=1,\cdots,t,j=1,\cdots,l$} \\
& g'_{k'}(4) \mbox{~~~ for $v=p_{j};j=1,\cdots,l$},
       \end{array} \right.
\end{equation}
for $k'=1,2,3,4$.\\
Clearly any vector in (\ref{e8}) is perpendicular to any other vector in (\ref{e6}) and (\ref{e15}) and satisfies the equation $$A_{S(F_{l}^{\{s,t\}})}g_{k'}=\lambda_{k'} g_{k'}, ~\mbox{for}~k'=1,2,3,4.$$
Thus the another set of eigenvalues $\lambda_{k'}$ $(k'=1,2,3,4)$ of $A_{S(F_{l}^{\{s,t\}})}$ can be obtained from the equation (\ref{e7}).\\
Let us construct $3(l-1)$ linearly independent vectors which are orthogonal to the previously obtained eigenvectors. The construction of these vectors are as follows. For $j'=1,\cdots,l-1$ we have
 \begin{equation}\label{e9}
f_{j'}(v)=\left\{ \begin{array}{rl}
& c_{1}\alpha^{j+j'-1} \mbox{~~~ for $v=v_{j};j=1,\cdots,l$} \\
& c_{2}\alpha^{j+j'-1} \mbox{~~~ for $v=u_{i}^{j};i=1,\cdots,t,j=1,\cdots,l$}\\
& c_{3}\alpha^{j+j'-1} \mbox{~~~ for $v=p_{j};j=1,\cdots,l$},
       \end{array} \right.
\end{equation}
where $\alpha$ is the $l^{th}$ roots of unity and $c_{1}, c_{2}, c_{3}\in R$.
If $f_{j'}$ is an eigenvector of $A_{S(F_{n})}$ with an eigenvalue $\lambda$ then we have
 $$A_{S(F_{l}^{\{s,t\}})}f_{j'}=\lambda f_{j'},$$ which gives us the following set of equations.\\
$$
\left. \begin{array}{rl}
&\lambda \alpha^{j}c_{1}=(\alpha^{j-1}+\alpha^{j})\frac{t(k-2)}{k-1}c_{2}+(\alpha^{j-1}+\alpha^{j+1})\frac{k-2}{k-1}c_{1}+(\alpha^{j-1}+\alpha^{j})c_{3}. \\
&\lambda \alpha^{j}c_{2}=\frac{(t-1)(k-2)}{k-1}\alpha^{j}c_{2}+(\alpha^{j}+\alpha^{j+1})\frac{k-2}{k-1}c_{1}+\alpha^{j}c_{3}. \\
&\lambda \alpha^{j}c_{3}=t\alpha^{j}c_{2}+(\alpha^{j}+\alpha^{j+1})c_{1}.
 \end{array} \right\}
$$
Solving the above equations we have the following equations for $\lambda$.\\
$(k-1)^{2}\lambda^{3}-(k-1)(k-2)\{(\alpha^{-1}+\alpha)+(t-1)\}\lambda^{2}-\{t(k-1)^{2}+(\alpha^{-1}+2+\alpha)(k-1)^{2}+(\alpha^{-1}+2+\alpha)t(k-2)^{2}-(\alpha^{-1}+\alpha)(t-1)(k-2)^{2}(k-4)\}\lambda-(\alpha^{-1}+2+\alpha)(k-1)(k-2)-2t(k-1)(k-2)=0$,\\
where $\alpha^{-1}+\alpha=2cos\frac{2\pi j'}{l}$, $j'=1,\cdots,l-1$.\\
Clearly any vector in (\ref{e9}) is orthogonal to any vector in (\ref{e6}) and (\ref{e15}) and any vector in (\ref{e8}) because
\begin{eqnarray*}
f_{j'}^{T}(v)g_{k'}(v)&=&0\times g'_{k'}(1)+c_{1}g'_{k'}(2)(1+\alpha+\cdots+\alpha^{l})+tc_{2}g'_{k'}(3)(1+\alpha+\cdots+\alpha^{l})\\&&+c_{3}g'_{k'}(4)(1+\alpha+\cdots+\alpha^{l})\\
&=&0.
\end{eqnarray*}
Thus we have our desired result.
\end{proof}
\subsection{Adjacency spectra of subdivision of a squid like hypergraph}
First we recall the definition of a squid from \cite{Hu}.
\begin{definition}
Let $H=(V,E)$ be a $k$-uniform hypergraph. If we can number the vertex set $V(H)$ as $V(H)=\{i_{1,1},\cdots,i_{1,k},\cdots,i_{k-1,1},\cdots,i_{k-1,k},i_{k}\}$ such that the set of hyperedges $E(H)=\{\{i_{1,1},\cdots,i_{1,k}\},\cdots,\{i_{k-1,1},\cdots,
i_{k-1,k}\},\{i_{1,1},\cdots,i_{k-1,1},i_{k}\}\}$, then $H$ is a squid.
\end{definition}
Now we define a squid like hypergraph.
\noindent
\textbf{Squid like hypergraph}: Let $H=(V,E)$ be a $k$-uniform hypergraph. Here we take the vertex set $V(H)$ as $V(H)=\{i_{1,1},\cdots,i_{1,k},\cdots,i_{k,1},\cdots,i_{k,k}\}$ such that the set of hyperedges $E(H)=\{\{i_{1,1},\cdots,i_{1,k}\},\cdots,\{i_{k,1},\cdots,
i_{k,k}\},\{i_{1,1},\cdots,i_{k,1}\}\}$.
\begin{figure}
\centering
\includegraphics[width=2.0in]{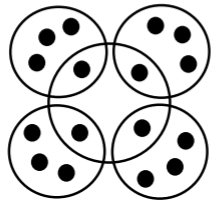}
\caption{Squid like hypergraph}
\end{figure}
\begin{theorem}\label{t6}
Let $S_{k}$ be the $k$-uniform squid like hypergraph. If $k\geq2$ is an integer, then the spectrum of $S(S_{k})$ consists of:
\begin{enumerate}[(i)]
\item the eigenvalue $-\frac{k-2}{k-1}$ with multiplicity $k(k-2)$,
\item the three roots of the equation \\ $(k-1)^{2}\lambda^{3}-(k-1)^{3}\lambda^{2}-\{(k-2)^{2}+(k-1)^{2}+2(k-2)^{3}+(k-1)^{3}\}\lambda+(k-1)(k-2)-2(k-1)^{2}(k-2)=0$, each with multiplicity $(k-1)$, and
\item the four roots of the equation \\
$(k-1)\lambda^{4}-\{(k-2)^{2}+(k-2)(k-1)\}\lambda^{3}-\{(k-1)^{2}+(k-1)(k+1)+(k-2)^{2}-(k-2)^3\}\lambda^{2}+\{(k-1)^{2}(k-2)+(k-2)^{2}(k+1)-2(k-1)(k-2)\}\lambda+k(k-1)^{2}=0$.
\end{enumerate}
\end{theorem}
\begin{proof}
Let $E(S_{k})=\{e_{1},\cdots,e_{k}\}$, let $U_{j}=\{u_{1}^{j},\cdots,u_{k-1}^{j}\}$ be the vertices of degree one in $e_{j}$, $W=\{w_{1},\cdots,w_{k}\}$ be the vertices of degree two in center hyperedge of $S_{k}$, let $p$ be the new vertex and $Q=\{q_{1},\cdots,q_{k}\}$ be the new vertices introduced while subdividing the central hyperedge and petal hyperedges of $S(S_{k})$ respectively.\\
If we partition the vertices as $\{w_{1},\cdots,w_{k}\}\cup\{u_{1}^{1},\cdots,u_{k-1}^{1}\}\cup\cdots\cup\{u_{1}^{k},\cdots,u_{k-1}^{k}\}\cup p \cup\{q_{1},\cdots,q_{k}\},$
then the adjacency matrix of $S(S_{k})$ can be written as:\\
\begin{eqnarray*}
A_{S(S_{k})}&=&\frac{1}{k-1}\begin{pmatrix}
  (k-2)(J_{k}-I_{k}) & (k-2)\mathbf{1}_{k-1}^{T}\otimes I_{k} & (k-1)\mathbf{1}_{k} & (k-1)I_{k} \\
  (k-2)\mathbf{1}_{k-1}\otimes I_{k} & (k-2)(J_{k-1}-I_{k-1})\otimes I_{k} & O_{k(k-1)\times 1} & (k-1)\mathbf{1}_{k-1}\otimes I_{k}\\
  (k-1)\mathbf{1}_{k}^{T} & O_{1\times k(k-1)} & O_{1} & O_{1\times k}\\
  (k-1)I_{k} & (k-1)\mathbf{1}_{k-1}^{T}\otimes I_{k} & O_{k\times 1} & O_{k}
  \end{pmatrix}.
\end{eqnarray*}
We can construct the following family of $k(k-2)$ linearly independent vectors.
\begin{equation}\label{e10}
 x_{i}^{\{e_{j}\}}(v)=\left\{ \begin{array}{rl}
& 1  \mbox{~~~for $v=u_{1}^{j}$}\\
& -1 \mbox{~~~ for $v=u_{i}^{j}$} \\
& 0 \mbox{~~~elsewhere},
       \end{array} \right.
\end{equation} for $2\leq i\leq t$.
Clearly $$A_{S(S_{k})}x_{i}^{\{e_{j}\}}=-\frac{k-2}{k-1}x_{i}^{\{e_{j}\}}.$$
Repeating this construction for the other hyperedges of $S(S_{k})$, we obtain $k(k-2)$ linearly independent eigenvectors, associated with the eigenvalue $\lambda=-\frac{k-2}{k-1}$.\\
Let us construct $3(k-1)$ linearly independent vectors which are orthogonal to the previously obtained eigenvectors. The construction of these vectors are as follows. For $j=2,\cdots,k$ we have
 \begin{equation}\label{e11}
f_{j}(v)=\left\{ \begin{array}{rl}
& c_{1} \mbox{~~~for $v=w_{1}$}\\
& -c_{1} \mbox{~~~for $v=w_{j}$}\\
& c_{2} \mbox{~~~ for $v=u_{i}^{1};i=1,\cdots,t$}\\
& -c_{2} \mbox{~~~ for $v=u_{i}^{j};i=1,\cdots,t$}\\
& c_{3} \mbox{~~~for $v=q_{1}$}\\
& -c_{3} \mbox{~~~for $v=q_{j}$}\\
& 0 \mbox{~~~elsewhere},
       \end{array} \right.
\end{equation}
where $c_{1}, c_{2}, c_{3}\in R$. If $f_{j}$ is an eigenvector of $A_{S(S_{k})}$ with an eigenvalue $\lambda$ then we have
 $$A_{S(S_{k})}f_{j}=\lambda f_{j},$$ which provides us the following set of equations.
$$
\left. \begin{array}{rl}
& \lambda c_{1}=(k-2)c_{2}-\frac{k-2}{k-1}c_{1}+c_{3}, \\
& \lambda c_{2}=\frac{(k-2)^{2}}{k-1}c_{2}+\frac{k-2}{k-1}c_{1}+c_{3}, \\
& \lambda c_{3}=(k-1)c_{2}+c_{1},
       \end{array} \right\}
$$\\
for $j=2,\cdots,k$. Eliminating $c_{1}$, $c_{2}$ and $c_{3}$ we have the following equation for the eigenvalue $\lambda$,\\
$(k-1)^{2}\lambda^{3}-(k-1)^{3}\lambda^{2}-\{(k-2)^{2}+(k-1)^{2}+2(k-2)^{3}+(k-1)^{3}\}\lambda+(k-1)(k-2)-2(k-1)^{2}(k-2)=0$, for $j=2,\cdots,k$.\\
Let $\{W,U,p,Q\}$ be a partition of $V(S(S_{k}))$. Clearly this is an equitable partition of $V(S(S_{k}))$ where the quotient matrix $Q$ is defined as:
\begin{eqnarray*}
Q&=&\frac{1}{(k-1)}\begin{pmatrix}
(k-1)(k-2) & (k-1)(k-2) & (k-1) & (k-1)\\
  (k-2) & (k-2)^2 & O & (k-1) \\
  k(k-1) & O & O & O\\
  (k-1) & (k-1)^2 & O & O
  \end{pmatrix}
\end{eqnarray*}
Let $(\lambda_{k'},g'_{k'})$ be an eigenpair of $Q$ for $k'=1,2,3,4$. Using the matrix $Q$ we get $\lambda_{k'}$ $(k'=1,2,3,4)$ are the roots of the following equations.
\begin{eqnarray}\label{e12}
(k-1)\lambda^{4}-\{(k-2)^{2}+(k-2)(k-1)\}\lambda^{3}-\{(k-1)^{2}+(k-1)(k+1)+(k-2)^{2}-(k-2)^3\}\lambda^{2}\notag\\+\{(k-1)^{2}(k-2)+(k-2)^{2}(k+1)-2(k-1)(k-2)\}\lambda+k(k-1)^{2}=0.
\end{eqnarray}
Let us define four vectors $g_{k'}$ for $k'=1,2,3,4$ as follows
\begin{equation}\label{e13}
g_{k'}(v)=\left\{ \begin{array}{rl}
& g'_{k'}(1) \mbox{~~~for $v=w_{j};j=1,\cdots,k$}\\
& g'_{k'}(2) \mbox{~~~ for $v=u_{i}^{j};i=1,\cdots,k-1,j=1,\cdots,k$} \\
& g'_{k'}(3) \mbox{~~~ for $v=p$} \\
& g'_{k'}(4) \mbox{~~~for $v=q_{j};j=1,\cdots,k$}.
       \end{array} \right.
\end{equation}
Clearly any vector in (\ref{e13}) is orthogonal to any other vector in (\ref{e10}) and (\ref{e11}) and satisfies the equation $$A_{S(S_{k})}g_{k'}=\lambda_{k'}g_{k'}, ~\mbox{for}~ k'=1,2,3,4.$$
Thus $\lambda_{k'}$ $(k'=1,2,3,4)$ are the remaining four eigenvalues of $A_{S(S_{k})}$ which can be derived from the equation (\ref{e12}). Thus we get our desired results.
\end{proof}

\section{Construction of cospectral hypergraphs}
The study of non-isomorphic cospectral hypergraphs is important since it reveals which hypergraph properties cannot be deduced from their spectra. Two hypergraphs are said to be $A$-cospectral if their adjacency matrices have the same spectrum. In \cite{Sar}, Sarkar and Banerjee gave an example of non-isomorphic $A$-cospectral hypergraph and constructed infinitely many pairs of non-isomorphic $A$-cospectral hypergraphs using vertex corona. Abiad and Khramova \cite{Abi} proposed a new method for constructing non-isomorphic uniform hypergraphs which are cospectral with respect to their adjacency matrices.
In this section we construct several classes of non-regular non-isomorphic $A$-cospectral hypergraphs using subdivision. 
\begin{theorem}\label{t8}
Let $H$, $H^{'}$ be two regular non-isomorphic hypergraphs. If $H$ and $H^{'}$ are $A$-cospectral then $S(H)$ and $S(H^{'})$ are non-isomorphic $A$-cospectral.
\end{theorem}
\begin{proof}
The Theorem \ref{t1} shows that the adjacency spectrum of subdivision hypergraph only depends on regularity, uniformity and the adjacency spectrum of the original hypergraph. So if we take regular non-isomorphic $A$-cospectral hypergraphs then we can construct non-regular non-isomorphic $A$-cospectral hypergraphs. Thus the results follows.
\end{proof}
Similarly, using Theorem \ref{t2} we have the following theorem.
\begin{theorem}\label{t7}
Let $G$, $G^{'}$ be regular non-isomorphic graphs. If $G$ and $G^{'}$ are $A$-cospectral then $S(G^{k})$ and $S(G^{'k})$ are non-isomorphic $A$-cospectral.
\end{theorem}
non-isomorphic $A$-cospectral regular graphs are readily available in the literature, for example, see \cite{Van}.

\section*{Acknowledgement}
The author is thankful to Samiron Parui for fruitful discussions.

\end{document}